\documentclass[12pt, reqno]{amsart}
\usepackage{amsmath, amsthm, amscd, amsfonts, amssymb, graphicx, color}
\usepackage[bookmarksnumbered, colorlinks, plainpages]{hyperref}
\hypersetup{colorlinks=true,linkcolor=red, anchorcolor=green, citecolor=cyan, urlcolor=red, filecolor=magenta, pdftoolbar=true}

\usepackage{tensor}

\newtheorem{theorem}{Theorem}[section]
\newtheorem{lemma}[theorem]{Lemma}
\newtheorem{proposition}[theorem]{Proposition}
\newtheorem{corollary}[theorem]{Corollary}
\theoremstyle{definition}
\newtheorem{definition}[theorem]{Definition}
\newtheorem{example}[theorem]{Example}

\theoremstyle{remark}
\newtheorem{remark}[theorem]{Remark}
\numberwithin{equation}{section}

\usepackage{fullpage}

\begin{document}

\setcounter{page}{1}

\title[FPT on $G$-metric spaces]{New fixed point results on $G$-metric spaces.}

\author[Ya\'e Olatoundji Gaba]{Ya\'e Olatoundji Gaba$^{1,*}$}

\address{$^{1}$Department of Mathematics and Applied Mathematics, University of Cape Town, South Africa.}
\email{\textcolor[rgb]{0.00,0.00,0.84}{gabayae2@gmail.com
}}

\subjclass[2010]{Primary 47H05; Secondary 47H09, 47H10.}

\keywords{$G$-metric, fixed point, orbitally continuous.}

\begin{abstract}
In this note, we discuss some fixed point theorems for contractive self mappings defined on a $G$-metric spaces. More precisely, we give fised point theorems for mappings with a contractive iterate at a point.
\end{abstract} 

\maketitle

\section{Introduction and preliminaries}

In recent years, numerous generalizations of the 
Banach contraction principle have appeared in the literature and the authors have introduced mappings of different contractive kind and studied the existence of related fixed points. The concept of metric space, as a convenient framework in fixed point theory, has been generalized in several directions. Some of such generalizations are $G$-metric spaces. Although a $G$-metric space is topologically equivalent to a metric space, both spaces are ``isometrically" distinct. Many fixed point in $G$-metric spaces appear in the litterature and the works by Jleli\cite{j}, Kadelburg\cite{r}, Mohanta\cite{s}, Mustafa et al. (\cite{mustafa1, mustafa3, mustafa4,mustafa2}), Patil\cite{p}, Gaba\cite{Gaba1,Gaba5} and many more, are leading results on the subjesct. The aim of this paper is to generalize, unify, and extend some theorems of well-known authors such as of \'{C}iri\'{c}\cite{ciric}, Jachymaski\cite{ja1}, Rhoades \cite{ro}, from metric spaces to $G$-metric spaces.

\vspace*{0.2cm}

The basic concepts and notations attached to the idea of $G$-metric space
can be read extensively in \cite{Mustafa} but for the convenience of the reader, we here recall the most important ones.

\begin{definition} (Compare \cite[Definition 3]{Mustafa})
Let $X$ be a nonempty set, and let the function $G:X\times X\times X \to [0,\infty)$ satisfy the following properties:
\begin{itemize}
\item[(G1)] $G(x,y,z)=0$ if $x=y=z$ whenever $x,y,z\in X$;
\item[(G2)] $G(x,x,y)>0$ whenever $x,y\in X$ with $x\neq y$;
\item[(G3)] $G(x,x,y)\leq G(x,y,z) $ whenever $x,y,z\in X$ with $z\neq y$;
\item[(G4)] $G(x,y,z)= G(x,z,y)=G(y,z,x)=\ldots$, (symmetry in all three variables);

\item[(G5)]
$$G(x,y,z) \leq [G(x,a,a)+G(a,y,z)]$$ for any points $x,y,z,a\in X$.
\end{itemize}
Then $(X,G)$ is called a \textbf{$G$-metric space}.

\end{definition}

\begin{proposition}\label{prop1} (Compare \cite[Proposition 6]{Mustafa})
Let $(X,G)$ be a $G$-metric space. Define on $X$ the metric  $d_G$ by $d_G(x,y)= G(x,y,y)+G(x,x,y)$ whenever $x,y \in X$. Then for a sequence $(x_n) \subseteq X$, the following are equivalent
\begin{itemize}
\item[(i)] $(x_n)$ is $G$-convergent to $x\in X$ .

\item[(ii)] $\lim_{n,m \to \infty}G(x,x_n,x_m)=0.$

\item[(iii)]  $\lim_{n \to \infty}d_G(x_n,x)=0$.

\item[(iv)]$\lim_{n \to \infty}G(x,x_n,x_n)=0.$ 

\item[(v)]$\lim_{n \to \infty}G(x_n,x,x)=0.$ 
\end{itemize}

\end{proposition}

\begin{proposition}(Compare \cite[Proposition 9]{Mustafa})

In a $G$-metric space $(X,G)$, the following are equivalent
\begin{itemize}
\item[(i)] The sequence $(x_n) \subseteq X$ is $G$-Cauchy.

\item[(ii)] For each $\varepsilon >0$ there exists $N \in \mathbb{N}$ such that $G(x_n,x_m,x_m)< \varepsilon$ for all $m,n\geq N$.

\item[(iii)] $(x_n)$ is a Cauchy sequence in the metric space $(X, d_G)$.
\end{itemize}

\end{proposition}

\begin{definition} (Compare \cite[Definition 4]{Mustafa})
 A $G$-metric space $(X,G,K)$ is said to be symmetric if 
 $$G(x,y,y) = G(x,x,y), \ \text{for all } x,y \in X.$$

\end{definition}

\begin{definition} (Compare \cite[Definition 9]{Mustafa})
 A $G$-metric space $(X,G)$ is said to be $G$-complete if every $G$-Cauchy sequence in $(X,G)$ is $G$-convergent. 

\end{definition}

We shall also make use of Proposition 1 from \cite{Mustafa}.

\section{Main results}


\begin{definition} 
A self mapping $T$ defined on a $G$-metric space $(X,G)$ is said to be \textbf{orbitally continuous} if and only if $\lim_{i\to \infty}T^{n_i}x=x^* \in X$ implies $Tx^*=\lim_{i\to \infty}TT^{n_i}x$.
\end{definition}

\begin{definition}
Let $T$ be a self mapping defined on a $G$-metric space $(X,G)$. The space $(X,G)$ is said to be $T$-\textbf{orbitally complete} if only if for any $a\in X$ every $G$-Cauchy sequence wchich is contained in $I(a,T):=\{a,Ta,T^2a,T^3a, \cdots\}$ $G$-converges in $X$.
\end{definition}

\vspace*{0.25cm}

\subsection{First results}

\begin{theorem}\label{thm1}
Let $(X,G)$ be a symmetric $G$-metric space and $T$ a self mapping on $X$. If $X$ is $T$-orbitally complete and $T$ is an orbitally continous map which is injective and satisfies

\begin{align}\label{condition}
G(Tx,Ty,Tz) < & q. \max \{ G(x,y,z),a(x,y,z)G(Tx,y,z)G(x,Ty,z)G(x,y,Tz),\nonumber \\
& [G(x,y,z)G(Tx,Ty,Tz)]^{-1}G(x,Tx,Tx)G(y,Ty,Ty)G(z,Tz,Tz) \}
\end{align}

for all $x,y,z\in X, x\neq y$ and $q<1$, where $a(x,y,z)$ is a non-negative real valued function. Then for each $x\in X$, $\lim_{n\to\infty}T^nx=u_x\in X$ and $Tu_x=u_x$. If in addition $a(x,y,z) \leq [G(x,y,z)G(Tx,Ty,Tz)]^{-1},$ then $T$ has a unique fixed point.
\end{theorem}

\begin{proof}
Let $x\in X$ and assume that $Tx\neq x$. Then by \eqref{condition}, we have

\begin{align*}
G(Tx,T^2x,T^2x) < &  q \max \{ G(x,Tx,Tx),0,\\
& [G(x,Tx,Tx)G(Tx,T^2x,T^2x)]^{-1}G(x,Tx,Tx)G(Tx,T^2x,T^2x)G(Tx,T^2x,T^2x) \}\\
               = & q \max \{ G(x,Tx,Tx),G(Tx,T^2x,T^2x) \}
\end{align*}

and hence $$ G(Tx,T^2x,T^2x)<   q. G(x,Tx,Tx).$$

\newpage

If $Tx=x$, then $ G(Tx,T^2x,T^2x)=0\leq q.G(x,Tx,Tx)$, and therefore

\begin{align}\label{usual}
 G(Tx,T^2x,T^2x)\leq q.G(x,Tx,Tx).
\end{align}

Again, we have 

\begin{align}\label{usual1}
 G(T^2x,T^3x,T^3x)\leq q.G(Tx,T^2x,T^2x)\leq q^2.G(x,Tx,Tx).
\end{align}

By usual procedure from \eqref{usual} and \eqref{usual1}, it follws that for any $p \in \mathbb{N}$
\begin{align*}
 G(T^nx,T^{n+p}x,T^{n+p}x)\leq \frac{q^n}{1-q} G(x,Tx,Tx).
\end{align*}
Since $q<1$, it follows that $\{x,Tx,T^2x,\cdots,T^nx,\cdots\}$ is a $G$-Cauchy sequence. By $T$-orbitally completeness of $X$, there exits $u_x\in X$ such that $T^nx$ $G$-converges to $u_x.$
Moreover, since $T$ is orbitally continuous, we have $$Tu_x=\lim_{n\to \infty} T^{n+1}x=u_x.$$

Hence the first part of the Theorem is proved.

Let now  $a(x,y,z) \leq [G(x,y,z)G(Tx,Ty,Tz)]^{-1}$ and suppose that $u=Tu$, $v=Tv$ and $u\neq v$. Then

\begin{align*}
G(u,v,v) = G(Tu,Tv,Tv) &< q. \max \{G(u,v,v), 0,[G(u,v,v)G(u,v,v)]^{-1}[G(u,v,v)]^3 \}\\
&= q G(u,v,v),
\end{align*}

which is a contradiction with $q<1.$
This completes the proof.
\end{proof}

\vspace*{0.5cm}

\begin{theorem}\label{thm2}
Let $(X,G)$ be a symmetric $G$-metric space and $T$ an orbitally continous self mapping on $X$ which is injective and satisfies
\begin{align}\label{condition1}
G(Tx,Ty,Tz) < &  \max \{ G(x,y,z),a(x,y,z)G(Tx,y,z)G(x,Ty,z)G(x,y,Tz),\nonumber \\
& [G(x,y,z)G(Tx,Ty,Tz)]^{-1}G(x,Tx,Tx)G(y,Ty,Ty)G(z,Tz,Tz)\}
\end{align}

for all $x,y,z\in X, x\neq y$, where $a(x,y,z)$ is a non-negative real valued function\footnote{This is the case of maps which satisfy \eqref{condition} with $q=1$.}. Then if for some $x_0\in X$, $\{T^nx_0\}$ has a cluster point $u\in X$ then $u$ is a fixed point of $T$ and $\{T^nx_0\}$ $G$-converges to $u.$
\end{theorem}

\begin{proof}\hspace*{0.3cm}

\vspace*{0.3cm}

If $T^nx_0= T^{n+1}x_0$ for some $n\in \mathbb{N}$, then $u=T^nx_0$ and the proof is complete.

Assume now that $T^nx_0\neq T^{n-1}x_0$ for all $n=1,2,3,\cdots,$ and let $$\lim_{i\to \infty} T^{n_i}x_0=u.$$

\newpage

Then by \eqref{condition1}, we have

\begin{align*}
G(T^{n}x_0,T^{n+1}x_0,T^{n+1}x_0)& < \max \{ G(T^{n-1}x_0,T^{n}x_0,T^{n}x_0),0,\\
&[G(T^{n-1}x_0,T^{n}x_0,T^{n}x_0)G(T^{n}x_0,T^{n+1}x_0,T^{n+1}x_0)]^{-1}\\
&G(T^{n-1}x_0,T^{n}x_0,T^{n}x_0)G(T^{n}x_0,T^{n+1}x_0,T^{n+1}x_0)\\
& G(T^{n}x_0,T^{n+1}x_0,T^{n+1}x_0)
\}.
\end{align*}
Hence 
\[
G(T^{n}x_0,T^{n+1}x_0,T^{n+1}x_0) < G(T^{n-1}x_0,T^{n}x_0,T^{n}x_0),
\]
as $G(T^{n}x_0,T^{n+1}x_0,T^{n+1}x_0) < G(T^{n}x_0,T^{n+1}x_0,T^{n+1}x_0)$ is impossible. Therefore, the sequence
\[
\{ G(T^{n}x_0,T^{n+1}x_0,T^{n+1}x_0)\}
\]
is a decreasing sequence of positive reals and hence convergent. Since $\lim_{i\to \infty} T^{n_i}x_0=u$ and $T$ is orbitally continuous, it follows that $Tu= \lim_{i\to \infty} T^{n_i+1}x_0$, $T^2u=\lim_{i\to \infty} T^{n_i+2}x_0$ and 

\begin{equation}\label{3}
\lim_{i\to \infty}G(T^{n_i}x_0,T^{n_i+1}x_0,T^{n_i+1}x_0)= G(u,Tu,Tu),
\end{equation}

\begin{equation}\label{4}
\lim_{i\to \infty}G(T^{n_i+1}x_0,T^{n_i+2}x_0,T^{n_i+2}x_0)= G(Tu,T^2u,T^2u).
\end{equation}

Since $\{ G(T^{n}x_0,T^{n+1}x_0,T^{n+1}x_0)\}$ is a convergent sequence and $\{G(T^{n_i}x_0,T^{n_i+1}x_0,T^{n_i+1}x_0) \}$  and $\{G(T^{n_i+1}x_0,T^{n_i+2}x_0,T^{n_i+2}x_0) \}$ are two of its subsequences, it follows from \eqref{3} and \eqref{4} that 

\[
\lim_{n\to \infty}G(T^{n}x_0,T^{n+1}x_0,T^{n+1}x_0)=G(u,Tu,Tu)=G(Tu,T^2u,T^2u).
\]

Therefore, we have

\begin{equation}\label{5}
G(u,Tu,Tu)=G(Tu,T^2,T^2u).
\end{equation}

If we assume that $u\neq Tu$, then by \eqref{condition1} we obtain 
\[
G(Tu,T^2u,T^2u)< G(u,Tu,Tu),
\]
which contradicts \eqref{5}. Therefore, we confidently conclude that $Tu=u$.
\end{proof}

\begin{corollary}
Let $X$ be a compact $G$-metric space and $T$ an injective and orbitally continous self mapping on $X$. If $T$ satisfies \eqref{condition1}, then for each $x\in X$, we have $\lim_{n\to\infty}T^nx=u_x\in X$ for some $=u_x\in X$ and $Tu_x=u_x$.   If in addition $a(x,y,z) \leq [G(x,y,z)G(Tx,Ty,Tz)]^{-1},$ then $T$ has a unique fixed point.
\end{corollary}

We now introduce the family $\mathcal{H}$ of functions that we shall use for the next result. For the terminology \textit{upper semicontinuous}, we shall use the short form \textit{usc}.
We have

\[
h \in \mathcal{H} \Longleftrightarrow h:([0,\infty))^3\to[0,\infty),\ h \text{ is usc and nondecreasing in each variable.} 
\]

\begin{lemma}\label{lemma1}
Let $h\in \mathcal{H}$, set $g(t)=h(t,t,t)$. $$ \text{For every } t>0,\ \  g(t)<t \text{ if and only if } \lim_{n\to \infty} g^n(t)=0.$$
\end{lemma}

\begin{proof} \hspace*{0.3cm}

\vspace{0.3cm}

\textbf{Necessary condition:} \hspace*{0.3cm} 

\vspace{0.15cm}

 Since $h$ is usc, then so is $g$. Assume now that $lim_{n\to \infty} g^n(t)=K\neq 0.$ Then
 
 \[
 K = \lim_{n\to \infty} g^{n+1}(t) \leq g( \lim_{n\to \infty} g^n(t))= g(K)< K,
 \]
--a contradiction, therefore $K=0.$

\vspace{0.5cm}

\textbf{Sufficient condition:} \hspace*{0.3cm} 

\vspace{0.15cm}

 Since $h$ is non-decreasing, then so is $g$. Given that $\lim_{n\to \infty} g^n(t)=0,$ for every $t>0,$ assume that $g(t)>t$ for some $t^*>0$. Then $g^n(t^*)>t^*$ for $n=1,2,3,\cdots.$ Thus $$\lim_{n\to \infty} g^n(t^*)\geq t^*>0,$$
--a contradiction.

Moreover, if $g(t^*)=t^*$, then

$$\lim_{n\to \infty} g^n(t^*)= t^*>0.$$

Hence, for all $t>0, \ g(t)<t.$
\end{proof}

We now propose a generalization of Theorem \ref{thm1}.

\begin{theorem}\label{thm3}
Let $(X,G)$ be a symmetric $G$-metric space. If $X$ is $T$-orbitally complete and $T$ is an orbitally continous self map on $X$ which is injective and satisfies

\begin{align}\label{conditionthm3}
G(Tx,Ty,Tz) \leq & h( G(x,y,z),\nonumber\\
& [G(x,y,z)G(Tx,Ty,Tz)]^{-1}G(x,Tx,Tx)G(y,Ty,Ty)G(z,Tz,Tz),\nonumber\\
& a(x,y,z)G(Tx,y,z)G(x,Ty,z)G(x,y,Tz)) 
\end{align}
for all $x,y,z\in X, x\neq y$, where $a(x,y,z)$ is a non-negative real valued function such that $a(x,y,y)=0$ and for some $h\in \mathcal{H}$ such that $h(t,t,t)<t$ for all $t>0$. Then for each $x\in X$, $\lim_{n\to\infty}T^nx=u_x\in X$ and $Tu_x=u_x$. If in addition $a(x,y,z) \leq [G(x,y,z)G(Tx,Ty,Tz)]^{-1},$ then $T$ has a unique fixed point.

\end{theorem}

\begin{proof}
Let $x$ be any point in $X$. Suppose that $Tx\neq x$. 

From \eqref{conditionthm3}, we have:
\[
G(Tx,T^2x,T^2x) \leq h(G(x,Tx,Tx),G(Tx,T^2x,T^2x),0).
\]

Since $G(x,Tx,Tx)< G(Tx,T^2x,T^2x)$ leads to 
\[
G(Tx,T^2x,T^2x) \leq h(G(Tx,T^2x,T^2x),G(Tx,T^2x,T^2x),G(Tx,T^2x,T^2x))< G(Tx,T^2x,T^2x),
\]
--a contradiction, we conclude that 

\[
G(Tx,T^2x,T^2x) \leq G(x,Tx,Tx).
\]

\newpage 

Therefore, we get that

\begin{equation}\label{lem1}
G(Tx,T^2x,T^2x) \leq h(G(x,Tx,Tx),G(x,Tx,Tx),G(x,Tx,Tx))=g(G(x,Tx,Tx)).
\end{equation}

Similarly, we have that
\[
G(T^2x,T^3x,T^3x) \leq h(G(Tx,T^2x,T^2x),G(T^2x,T^3x,T^3x) ,0).
\]

and since $ G(T^2x,T^3x,T^3x) \leq G(Tx,T^2x,T^2x)$ one gets that 

\[
G(T^2x,T^3x,T^3x) \leq g(G(Tx,T^2x,T^2x)) \leq g^2(G(x,Tx,Tx)).
\]

More generally
\[
G(T^nx,T^{n+1}x,T^{n+1}x) \leq g^n(G(x,Tx,Tx)).
\]

From Lemma \ref{lemma1}, we know that $\lim_{n\to \infty} g^n(t)=0$ for $t>0$ and hence
\begin{align}\label{r3}
\lim_{n\to \infty}G(T^nx,T^{n+1}x,T^{n+1}x)=0.
\end{align}

Next, we prove that $\{T^nx\}$ is a $G$-Cauchy sequence, and to this aim, it is enough to prove that $\{T^{2n}x\}$ is a $G$-Cauchy sequence.

Let's set $G_n= G(x_n,x_{n+1},x_{n+1})$. Suppose now, by the way of contradiction, that $\{T^{2n}x\}$ is not a $G$-Cauchy sequence. Then there exists an $\varepsilon>0$ such that for each $2k,\ k\in \mathbb{N}$, there exist $2n(k)$ and $2m(k)$ with $2k\leq 2n(k) < 2m(k) $ such that 

\begin{align}\label{r4}
G(x_{2n(k)},x_{2m(k)},x_{2m(k)})> \varepsilon.
\end{align}
Moreover, let's assume that for each $2k,\ k\in \mathbb{N}$, $2m(k)$ is the least integer exceeding $2n(k)$ satisfying \eqref{r4}, that is 

\begin{align}\label{r5}
G(x_{2n(k)},x_{2m(k)-2},x_{2m(k)-2})\leq \varepsilon \ \text{     and  } \ G(x_{2n(k)},x_{2m(k)},x_{2m(k)})> \varepsilon.
\end{align}

Then, we have
\[
\varepsilon < G(x_{2n(k)},x_{2m(k)},x_{2m(k)}) \leq G(x_{2n(k)},x_{2m(k)-2},x_{2m(k)-2}) + G_{2m(k)-2}+ G_{2m(k)-1}.
\]

By \eqref{r3} and \eqref{r5}, we conclude that 

\begin{align}\label{6}
G(x_{2n(k)},x_{2m(k)},x_{2m(k)}) \longrightarrow \varepsilon \text{ as } k\to \infty.
\end{align}

From 
\[
G(x_{2n(k)},x_{2m(k)},x_{2m(k)}) \leq G(x_{2n(k)},x_{2m(k)-1},x_{2m(k)-1}) + G(x_{2m(k)-1},x_{2m(k)},x_{2m(k)})
\]
and 

\begin{align*}
G(x_{2n(k)},x_{2m(k)-1},x_{2m(k)-1}) & \leq  G(x_{2n(k)},x_{2m(k)},x_{2m(k)}) + G(x_{2m(k)},x_{2m(k)-1},x_{2m(k)-1}) \\
& = G(x_{2n(k)},x_{2m(k)},x_{2m(k)}) + G(x_{2m(k)-1},x_{2m(k)},x_{2m(k)}),
\end{align*}

we obtain that 

\[
|G(x_{2n(k)},x_{2m(k)-1},x_{2m(k)-1}) - G(x_{2n(k)},x_{2m(k)},x_{2m(k)}) | \leq G_{2m(k)-1}.
\]

\newpage

Simalrly, we can obtain 

\[
| G(x_{2n(k)+1},x_{2m(k)-1},x_{2m(k)-1})- G(x_{2n(k)},x_{2m(k)},x_{2m(k)})| \leq G_{2m(k)-1} + G_{2n(k)}.
\]

By \eqref{6}, as $k\to \infty,$ we have that

$$ G(x_{2n(k)},x_{2m(k)-1},x_{2m(k)-1}) \longrightarrow \varepsilon $$

and 

$$ G(x_{2n(k)+1},x_{2m(k)-1},x_{2m(k)-1}) \longrightarrow \varepsilon .$$

Setting $$p=G(x_{2n(k)},x_{2m(k)},x_{2m(k)}),\ \ q =G(x_{2n(k)},x_{2m(k)-1},x_{2m(k)-1}),$$
and $$ r=  G(x_{2n(k)+1},x_{2m(k)-1},x_{2m(k)-1}), $$ we get, using the trianle inequality that:

$$p \leq G_{2n(k)}+ G(x_{2n(k)+1},x_{2m(k)},x_{2m(k)}).$$

By \eqref{conditionthm3}

\begin{align*}
G(x_{2n(k)+1},x_{2m(k)},x_{2m(k)})\leq  &  h(q, (q.G(x_{2n(k)+1},x_{2m(k)},x_{2m(k)}))^{-1}G_{2n(k)}G_{2m(k)-1}G_{2m(k)-1},\\
& a(x_{2n(k)},x_{2m(k)-1},x_{2m(k)-1}).r.[G(x_{2n(k)},x_{2m(k)-1},      x_{2m(k)}]^2))\\
=   & h(q, (q.G(x_{2n(k)+1},x_{2m(k)},x_{2m(k)}))^{-1}G_{2n(k)}G_{2m(k)-1}G_{2m(k)-1},0) .
\end{align*}

Since $h$ is usc, as $n\to \infty,$ it follows that 

\[
\varepsilon \leq h(\varepsilon ,0,0 )\leq  h(\varepsilon ,\varepsilon,\varepsilon )<\varepsilon, 
\]
--a contradiction. Therefore, $\{T^nx\}$ is  $G$-Cauchy.

 By orbital completeness of $X$, there exits $u_x\in X$ such that $T^nx$ $G$-converges to $u_x.$
Moreover, since $T$ is orbitally continuous, we have $$Tu_x=\lim_{n\to \infty} T^{n+1}x=u_x.$$

Hence the first part of the Theorem is proved.

\vspace*{0.3cm}

Let now  $a(x,y,z) \leq [G(x,y,z)G(Tx,Ty,Tz)]^{-1}$ and suppose that $u=Tu$, $v=Tv$ and $u\neq v$. Then

\begin{align*}
G(u,v,v) = G(Tu,Tv,Tv) &< h (G(u,v,v), 0,[G(u,v,v)G(u,v,v)]^{-1}[G(u,v,v)]^3 )\\
&=  h (G(u,v,v), 0,G(u,v,v) )\\
&\leq  g (G(u,v,v)) \\
& <  G(u,v,v), 
\end{align*}
which is a contradiction unless $u=v$. This completes our proof.
\end{proof}

\begin{corollary}

Let $(X,G)$ be a symmetric $G$-metric space. If $T$ is an orbitally continous self map on $X$ which is injective and satisfies

\begin{align}\label{conditioncorthm3}
G(Tx,Ty,Tz) \leq & h( G(x,y,z),\nonumber\\
& [G(x,y,z)G(Tx,Ty,Tz)]^{-1}G(x,Tx,Tx)G(y,Ty,Ty)G(z,Tz,Tz),\nonumber\\
& a(x,y,z)G(Tx,y,z)G(x,Ty,z)G(x,y,Tz)) 
\end{align}
for all $x,y,z\in X, x\neq y$, where $a(x,y,z)$ is a non-negative real valued function and for some $h\in \mathcal{H}$ such that $h(t,t,t)<t$ for all $t>0$. Then if for some $x_0\in X$, $\{T^nx_0\}$ has a cluster point $u\in X$ then $u$ is a fixed point of $T$ and $\{T^nx_0\}$ $G$-converges to $u.$

\end{corollary}

We conclude this section by giving an example of a mapping that satisfies \eqref{condition} but not \eqref{conditionthm3}.

\begin{example}
Let $X=[0,\infty)$ with $G(x,y,z)= \max\{|x-y|,||y-z, |z-x|\}$. Define the mappings $T:x\to X$ by $Tx=x(x+1)^{-1}$ and $h:([0,\infty))^3\to[0,\infty)$ by $h(x,y,z)=x(x+1)^{-1}$. It is clear that $h$ satisfies all the conditions of Theorem \ref{conditionthm3}.
Furthermore, for any $e,b\in [0,\infty), e\neq b$
\[
|Te-Tb|= \frac{|e-b|}{1+e+b+eb} \leq \frac{|e-b|}{1+|e-b|}.
\]

So, for any $x,y,z\in [0,\infty), x\neq y$, we have
\begin{align*}
G(Tx,Ty,Tz)  \leq & \max \left\lbrace\frac{|x-y|}{1+|x-y|}, \frac{|x-z|}{1+|x-z|},\frac{|z-y|}{1+|z-y|} \right\rbrace \\
        =  & h(G(Tx,Ty,Tz),\\
        & [G(x,y,z)G(Tx,Ty,Tz)]^{-1}G(x,Tx,Tx)G(y,Ty,Ty)G(z,Tz,Tz),0),
\end{align*}
where $a(x,y,z)=0$. Hence \eqref{conditionthm3} holds.

Moreover, since $$\lim_{n\to \infty}T^nx=\lim_{n\to \infty}x(1+nx)^{-1}=0$$ implies that 
$$T0 = \lim_{n\to \infty} TT^nx,$$

$T$ is orbitaly continuous and $X$ is orbitally complete. It follows from Theorem \ref{conditionthm3} that $T$ has a unique fized point, which in this case is $0$.

However, $T$ does not satisfy \eqref{condition}. Indeed if it was the case, for some $q<1, $ and for all $x \in X, x\neq 0,$

\[
x(x+1)^{-1} = G(T0,Tx,Tx) < q .\max \{ x,0,0 \}=qx,
\]
which leads $(x+1)^{-1}<q$ for all $x \in X, x\neq 0$ and this is impossible. Hence $T$ does not satisfy \eqref{condition}.
\end{example}

\vspace*{0.25cm}

\subsection{Extensions}\hspace{0.5cm}

\vspace*{0.25cm}

This last section of the manuscript is devoted to some extensions of results from \'{C}iri\'{c}\cite{ciric} and Rhoades \cite{ro}. They present more general cases  of the results discussed in previous sections.

\begin{theorem}\label{extension}
Let $T$ be a self mapping on a symmetric $G$-metric space $(X,G)$ and $X$ be $T$-orbitally complete. If there exists an element $x^*\in X$ such that for any three elements $x,y,z \in I(x^*,T)$, at least one of the following is true:
\begin{enumerate}

\item[(i)]

\[
G(x,Tx,Tx)+G(y,Ty,Ty)+G(z,Tz,Tz) \leq \alpha G(x,y,z), \ 1\leq \alpha < 3
\]

\item[(ii)]
\[
G(x,Tx,Tx)+G(y,Ty,Ty)+G(z,Tz,Tz) \leq \beta [ G(Tx,y,z)+ G(x,Ty,z)+ G(x,y,Tz)], 
\]
$ 1/2 \leq \beta < 1$

\item[(iii)]
\begin{align*}
G(Tx,Ty,Tz) \leq & \delta \max \left\lbrace G(x,y,z), G(x,Tx,Tx),G(y,Ty,Ty),G(z,Tz,Tz),\right.\\
& \left. \frac{1}{4}[G(Tx,y,z)+ G(x,Ty,z)+ G(x,y,Tz)]\right\rbrace, 
\end{align*}
$0\leq \delta <1.$

\end{enumerate}

then $\{T^nx^*\}$ $G$-converges and $\xi = \lim_{n\to \infty}T^nx^*$ is a fixed of $T$.

\end{theorem}

\begin{proof}

Define the sequence $d_n$ via the sequence of iterates $u_n=T^nx^*$ as $d_n=G(u_n, u_{n+1},u_{n+1})$ ($u_0=x^*$), $n=0,1,2,\cdots.$ 

Now suppose that (i) is true for the triplet $(u_n, u_{n+1},u_{n+1})$.

Then 
\[
d_n+2d_{n+1} \leq \alpha d_n,
\]
i.e.

\begin{align}\label{rr1}
d_{n+1} \leq \frac{\alpha-1}{2}d_n.
\end{align}

Similarly, if (ii) and (iii) are true, then correspondly we obtain:

\begin{align}\label{rr2}
d_{n+1} \leq \frac{2\beta-1}{2-2\beta}d_n,
\end{align}

\begin{align}\label{rr4}
d_{n+1} \leq \delta d_n.
\end{align}

From \eqref{rr1}--\eqref{rr4}, we observe that

\begin{align}\label{rr5}
d_{n+1} \leq \lambda\ d_n
\end{align}
for all $n$, where 

\[
 \lambda = \min \left\lbrace \frac{\alpha-1}{2},\frac{2\beta-1}{2-2\beta} , \delta \right\rbrace < 1.
\]

We can easily show that $d_n \to 0$ as $n\to \infty$ 
and by usual procedure, we derive that $\{u_n\}$ is a $G$-Cauchy sequence. Since $X$ is $T$-orbitally complete, then the limit $\xi$ of the sequence $\{u_n\}\in I(x^*,T) \subseteq X$.

\vspace*{0.2cm}

Now, we show that $\xi$ is a fixed point of $T$. For the triplet $(u_n,\xi,\xi)$, at least one of the following holds:

\begin{align}\label{rr6}
d_{n}+2 G(\xi,T\xi,T\xi) \leq \alpha \ G(u_n,\xi,\xi),
\end{align}

\begin{align}\label{rr7}
d_{n}+2 G(\xi,T\xi,T\xi) \leq \beta [ G(u_{n+1},\xi,\xi)+2G(u_n,\xi,\xi) +2G(\xi,T\xi,T\xi)],
\end{align}

\begin{align}\label{rr9}
G(u_{n+1},T\xi,T\xi) \leq & \delta \max \left\lbrace G(u_{n},\xi,\xi), d_n, G(\xi,T\xi,T\xi),\right.\nonumber \\
& \left. \frac{1}{4} [G(u_{n+1},\xi,\xi)+2G(u_{n},\xi,\xi)+2G(\xi,T\xi,T\xi)] \right\rbrace.
\end{align}

As we proceed along the sequence $\{u_n\}$, we obtain infinite values of $n$, say $\{n_k\}$, such that at least one the relations \eqref{rr6}--\eqref{rr9} is satisfied by the triplet $(u_{n_k},\xi,\xi)$. Lettting $k\to \infty$, we derive

$$ G(\xi,T\xi,T\xi) \leq 0; \quad G(\xi,T\xi,T\xi) \leq 2\beta G(\xi,T\xi,T\xi),$$ 

\vspace*{0.10cm}
$$\  G(\xi,T\xi,T\xi)\leq 2\delta G(\xi,T\xi,T\xi)  ,$$

\vspace*{0.10cm}
in the case of \eqref{rr6}, \eqref{rr7} and \eqref{rr9} respectively. In all these cases, the conclusion is that $\xi$ is a fixed point of $T$.

\end{proof}

Using the same idea as in Theorem \ref{thm2}, we are inspired to give the following lemma.

\begin{corollary}
Let $T$ be an orbitally continuous self mapping on a $G$-metric space $(X,G)$. Assume that there exists an element $x^*\in X$ such that for any three elements $x,y,z \in I(x^*,T)$, at least one of the inequalities (i)--(ii) is true. Moreover if $\{T^nx^*\}$ has a cluster point $u\in X$ then $u$ is a fixed point of $T$ and $\{T^nx^*\}$ $G$-converges to $u.$

\end{corollary}

\begin{remark}
The fixed point result stated in Theorem \ref{extension} leads to the existence of a unique fixed point if the map $T$ satisfies only the condition (iv).
\end{remark}

\begin{theorem}
Let $T$ be the map as defined in Theorem \ref{extension} and assume that $T$ satisfies only one of the conditions (i)--(iii). If $T$ further satisfies at least one of the condtions

\begin{itemize}
\item[v)]
\[
G(\xi,Tx,Tx) < G(x,x,\xi)+G(x,Tx,Tx),
\]

\item[vi)]

\[
G(\xi,x,x) < G(\xi,Tx,Tx)+ G(x,Tx,Tx),\text{ for all }x\neq \xi,
\]

\end{itemize}
then the uniqueness of the fixed point is guaranteed.
\end{theorem}

\bibliographystyle{amsplain}

\end{document}